\RequirePackage{fix-cm}
\documentclass[smallextended]{svjour3}       % onecolumn (second format)
\smartqed  % flush right qed marks, e.g. at end of proof
\usepackage{graphicx}
\begin{document}

\title{On the Riesz Basisness of Systems Composed of Root Functions of Periodic Boundary Value Problems%\thanks{Grants or other notes
}

\titlerunning{Riesz basis}        % if too long for running head

\author{Alp Arslan K{\i}ra\c{c}}


\institute{Department of Mathematics, Faculty of Arts and Sciences, Pamukkale
University, 20070, Denizli, Turkey \\
              Tel.: +90-258-2963625\\
                            \email{aakirac@pau.edu.tr}           %  \\
}

\date{Received: date / Accepted: date}

\maketitle

\begin{abstract}
In this paper, we consider the nonself-adjoint Sturm-Liouville operator with $q\in L_{1}[0,1]$ and either periodic, or anti-periodic boundary conditions. We obtain necessary and
sufficient conditions for systems of root functions of these operators to be a Riesz basis in $L_{2}[0,1]$ in terms of the Fourier coefficients of $q$.
\end{abstract}

\section{Introduction}
Let $L$ be Sturm-Liouville operator generated in $L_{2}[0,1]$ by the expression
\begin{equation}  \label{1}
y^{\prime\prime}+(\lambda-q)y=0,
\end{equation}
either with the periodic boundary conditions
\begin{equation}\label{per}
y(1)=y(0),\qquad y^{\prime}(1)=y^{\prime}(0),
\end{equation}
or with the anti-periodic boundary conditions
\begin{equation}\label{antiper}
y(1)=-y(0),\qquad y^{\prime}(1)=-y^{\prime}(0).
\end{equation}
where $q$ is a complex-valued summable function on $[0,1]$.
We will consider only the periodic problem. The anti-periodic problem is completely similar. The operator $L$ is regular, but not strongly regular. It is well known \cite{Dunford,Mikhailov} that the system of root functions
of an ordinary differential operator with strongly regular
boundary conditions forms a Riesz basis in $L_{2}[0,1]$. Generally, the normalized eigenfunctions and
associated functions, that is, the root functions of the operator with only regular
boundary conditions do not form a Riesz basis. Nevertheless, Shkalikov \cite{Shkalikov-1,Shkalikov-2} showed
that the system of root functions of an ordinary differential operator with regular
boundary conditions forms a basis with parentheses. In \cite{Kerimov;Riesz}, they proved that under the conditions
\begin{equation}\label{kerimcond}
q(1)\neq q(0),\qquad q\in C^{(4)}[0,1]
\end{equation}
the system of
root functions of $L$ forms a Riesz basis of $L_{2}[0,1]$. A new approach in terms
of the Fourier coefficients of $q$ is due to
Dernek and Veliev \cite{DERNEK:VELIEV}. They proved that if the following conditions
\begin{equation}\label{velder}
q_{2m}\sim q_{-2m},\quad \lim_{m\rightarrow \infty}\frac{ln|m|}{mq_{2m}}=0,
\end{equation}
hold, then the root functions of $L$ form a Riesz basis in $L_{2}[0,1]$,where
\[
q_{m}=:(q\,,e^{i2m\pi
x})=:\int_{0}^{1}q(x)\,e^{-i2m\pi x}\,dx
\]
is the Fourier coefficient of $q$ and without loss of generality we always suppose that
$q_{0}=0$ and the notation $a_{m}\sim b_{m}$ means that there exist constants $c_{1}$, $c_{2}$ such that
$0<c_{1}<c_{2}$ and $c_{1}<|a_{m}/b_{m}|<c_{2}$ for all large $m$. Makin \cite{makin2006} extended this result as follows:

\emph{Let the first condition (\ref{velder}) hold. But the second condition  (\ref{velder}) is replaced by a less restrictive one: $q\in W_{1}^{s}[0,1]$,
\begin{equation}\label{makincond}
q^{(k)}(1)=q^{(k)}(0),\,\forall\, k=0,1,\ldots,s-1
\end{equation}
holds and $|q_{2m}|>cm^{-s-1}$ with some $c>0$ for large $m$, where $s$ is a nonnegative
integer. Then the root functions of the operator $L$ form a Riesz basis in $L_{2}[0,1]$}.

In addition, some conditions which imply that the system of root functions does not form a Riesz basis of $L_{2}[0,1]$ were established in \cite{makin2006} (see also \cite{mityagininst1dim2006,mityagintripoly2011kýsa,mityagintripoly2011}). In \cite{kýraçriesz}, we proved that the Riesz basis property is valid if the first condition (\ref{kerimcond}) holds, but the second is replaced by $q\in W_{1}^{1}[0,1]$.
 The results of Shkalilov and Veliev \cite{Veliev:Shkalikov}
are more general and  inclusive. The assertions in various forms concerning the Riesz basis property were proved. One of the basic results in the paper \cite{Veliev:Shkalikov} is the following statement:

\emph{
Let $p\geq 0$ be an arbitrary integer, $q\in W_{1}^{p}[0,1]$ and (\ref{makincond}) holds with some $s\leq p$, and let one of the following conditions hold:
\begin{equation}\label{velshakcond}
|q_{2m}|>\varepsilon m^{-s-1}\quad \textrm{or}\quad |q_{-2m}|>\varepsilon m^{-s-1} \quad\textrm{for all large $m$}
\end{equation}
with some $\varepsilon>0$. Then a normal system of root functions of the operator $L$ forms a Riesz basis if and only if $q_{2m}\sim q_{-2m}$.
}

Here, for large $m$, denote by $\Psi_{m,j}(x)$ for $j=1,2$ the normalized eigenfunctions corresponding to the simple eigenvalues $\lambda_{m,j}$. If the multiplicities of these eigenvalues equal to $2$, then the root subspace consists either of two eigenfunctions, or of Jordan chains comprising one eigenfunction and
one associated function. First, if the multiple eigenvalue $\lambda_{m,1}=\lambda_{m,2}$  has geometric multiplicity 2, we take the normalized eigenfunctions  $\Psi_{m,1}(x)$, $\Psi_{m,2}(x)$. Secondly, if there is one eigenfunction $\Psi_{m,1}(x)$ corresponding to  the multiple eigenvalue $\lambda_{m,1}=\lambda_{m,2}$, then we take the Jordan chain consisting of a normalized eigenfunction  $\Psi_{m,1}(x)$ and corresponding associated function denoted again by $\Psi_{m,2}(x)$ and orthogonal to $\Psi_{m,1}(x)$. Thus the system of root functions obtained in this way will be called a \emph{normal system}.

Moreover, for the other interesting results about the Riesz basis property of root functions of the periodic and anti-periodic problems, we refer in particular to \cite{mityagincriter1D2012,Gesztesy2012,mamedovmenken2008,mamedovLp2010} and \cite{velievnorder,veliev;arþiv}.

In this paper, we prove the following main result:
\begin{theorem}\label{mainthm}
Let $q\in L_{1}[0,1]$ be arbitrary complex-valued function and suppose that at least one of the conditions
\begin{equation}\label{maincon}
\lim_{m\rightarrow\infty}\frac{\rho(m)}{m\,q_{2m}}=0,\qquad\lim_{m\rightarrow\infty}\frac{\rho(m)}{m\,q_{-2m}}=0
\end{equation}
is satisfied, where $\rho(m)$, defined in (\ref{m-8}), is a common order of the Fourier coefficients $q_{2m}$ and $q_{-2m}$ of $q$.

Then a normal system of root functions of the operator $L$ forms a Riesz basis if and only if
\begin{equation}\label{mainq2mcond}
q_{2m}\sim q_{-2m}.
\end{equation}
\end{theorem}

This form of Theorem \ref{mainthm} is not novel (see, for example, \cite{Veliev:Shkalikov}). The novelty is in the term $\rho(m)$ defined in (\ref{m-8}) (see also Lemma \ref{L1}). Indeed, if we take $p=0$ in the Sobolev space $W_{1}^{p}[0,1]$ given above in \cite{Veliev:Shkalikov}, that is, if $q\in L_{1}[0,1]$ then the nonnegative integer $s$ in the conditions (\ref{velshakcond}) must be zero and  the assertion on the Riesz basis property remains valid with a less restrictive condition (\ref{maincon}) instead of (\ref{velshakcond}). For example, let $\rho(m)=o(m^{-1/2})$. If instead of (\ref{maincon}) we suppose that at least one of the following conditions holds
\[
|q_{2m}|>\varepsilon m^{-3/2}\quad \textrm{or}\quad |q_{-2m}|>\varepsilon m^{-3/2} \quad\textrm{for all large $m$}
\]
with some $\varepsilon$, then the assertion of Theorem \ref{mainthm} is obvious.

It is well known (see, e.g., \cite {Naimark}, Theorem 2 in page 64) that the
periodic eigenvalues $\lambda_{m,1}, \lambda_{m,2}$ are
located in pairs, satisfying the following
asymptotic formula
\[
    \lambda_{m,1}=\lambda_{m,2}+O(m^{1/2})=(2m\pi)^{2}+O(m^{1/2}),
\]
for $m\geq N$. Here, by $N\gg 1$, we denote large enough positive integer. From this
formula, the pair of the eigenvalues
$\{\lambda_{m,1}, \lambda_{m,2}\}$ is close to the number
$(2m\pi)^{2}$ and isolated from the remaining
eigenvalues of $L$ by a distance $m$. That is, we have, for $j=1,2$,
\begin{equation}  \label{dist1}
|\lambda_{m,j}-(2(m-k)\pi)^{2}|>|k||2m-k|>C\,m,
\end{equation}
for all $k\neq 0, 2m$ and $k\in \mathcal{Z}$, where $m\geq N$ and, here and in subsequent relations, $C
$ is some positive constant whose
exact value is not essential. For the potential $q=0$ and $m\geq 1$, clearly, the system
$\{e^{-i2m\pi x}, e^{i2m\pi x}\}$ is a basis of the
eigenspace corresponding to the eigenvalue
$(2m\pi)^{2}$ of the periodic boundary value problems.

Finally, let us state the following relevant theorem which will be used in the proof of Theorem \ref{mainthm}.
 \begin{theorem}\label{velievthm}(see \cite{Veliev:Shkalikov})
 The following assertions are equivalent:

i) a normal system of root functions of the operator L forms a Riesz basis in the space $L_{2}[0,1]$;

ii) the number of Jordan chains is finite and the relation
\begin{equation}\label{velievlem}
u_{m,j}\sim v_{m,j}
\end{equation}
holds for all indices $m$ and $j$ corresponding only to the simple eigenvalues $\lambda_{m,j}$ for $j=1,2$, where  $u_{m,j}$, $v_{m,j}$ are the Fourier coefficients defined in (\ref{uv});

iii) the number of Jordan chains is finite and the relation (\ref{velievlem}) for either $j=1$, or $j=2$ holds.
\end{theorem}

\section{Preliminaries}
The following well-known relation will be used to obtain, for large $m$, the asymptotic formulas for periodic eigenvalues
$\lambda_{m,j}$ corresponding to the normalized
eigenfunctions $\Psi_{m,j}(x)$:
\begin{equation}  \label{m1}
\Lambda_{m-k,j}(\Psi_{m,j},e^{i2(m-k)\pi
x})=(q\,\Psi_{m,j},e^{i2(m-k)\pi x}),
\end{equation}
where $\Lambda_{m-k,j}=\lambda_{m,j}-(2(m-k)\pi)^{2}$, $j=1,2$.
 From Lemma 1 in \cite{Melda.O}, we iterate (\ref{m1}) by using the following relations
\begin{equation}  \label{m2}
(q\,\Psi_{m,j},e^{i2m\pi
x})=\sum_{m_{1}=-\infty}^{\infty}q_{m_1}(\Psi_{m,j},e^{i
2(m-m_{1})\pi x}),
\end{equation}
\begin{equation}  \label{m3}
|(q\,\Psi_{m,j},e^{i2(m-m_{1})\pi
x})|< 3M,
\end{equation}
where for all $m\geq N$, $m_{1}\in\mathcal{Z}$ and $j=1,2$,
where $M=\sup_{m\in \mathcal{Z}}|q_{m}|$.

Hence, substituting (\ref{m2}) in (\ref{m1}) for $k=0$
and then isolating the terms with indices $m_{1}=0, 2m$,  we deduce, in view of $q_{0}=0$, that
\begin{equation}\label{m4}
 \Lambda_{m,j}(\Psi_{m,j},e^{i2m\pi
x})=q_{2m}(\Psi_{m,j},e^{-i2m\pi
x})+\sum_{m_{1}\neq 0,2m}q_{m_1}(\Psi_{m,j},e^{i
2(m-m_{1})\pi x}).
\end{equation}
First, we use (\ref{m1}) for $k=m_{1}$ in the right-hand side of (\ref{m4}). Then, considering (\ref{m2}) with the indices $m_{2}$ and isolating the terms with indices $m_{1}+m_{2}=0,2m$, we get
\begin{equation}\label{m412}
[\Lambda_{m,j}- a_{1}(\lambda_{m,j})]u_{m,j}=[q_{2m}+b_{1}(\lambda_{m,j})]v_{m,j}+R_{1}(m),
\end{equation}
by repeating this procedure once again, and
\begin{equation}\label{m4123}
[\Lambda_{m,j}- a_{1}(\lambda_{m,j})-a_{2}(\lambda_{m,j})]u_{m,j}=[q_{2m}+b_{1}(\lambda_{m,j})+b_{2}(\lambda_{m,j})]v_{m,j}+R_{2}(m),
\end{equation}
where $j=1,2,$
\begin{equation}\label{uv}
u_{m,j}=(\Psi_{m,j},e^{i2m\pi x}),\quad v_{m,j}=(\Psi_{m,j},e^{-i2m\pi x}),
\end{equation}
\[
a_{1}(\lambda_{m,j})=\sum_{m_{1}}\frac{q_{m_{1}}q_{-m_{1}}}{\Lambda_{m-m_{1},j}},\quad
a_{2}(\lambda_{m,j})=\sum_{m_{1},m_{2}}\frac{q_{m_{1}}q_{m_{2}}q_{-m_{1}-m_{2}}}{\Lambda_{m-m_{1},j}\,\Lambda_{m-m_{1}-m_{2},j}},
\]
\begin{equation}\label{B1}
b_{1}(\lambda_{m,j})=\sum_{m_{1}}\frac{q_{m_{1}}q_{2m-m_{1}}}{\Lambda_{m-m_{1},j}},\quad
b_{2}(\lambda_{m,j})=\sum_{m_{1},m_{2}}\frac{q_{m_{1}}q_{m_{2}}q_{2m-m_{1}-m_{2}}}{\Lambda_{m-m_{1},j}\,\Lambda_{m-m_{1}-m_{2},j}},
\end{equation}
\begin{equation}\label{R}
  R_{1}(m)=\sum_{m_{1},m_{2}}\frac{q_{m_{1}}q_{m_{2}}(q\,\Psi_{m,j},e^{i2(m-m_{1}-m_{2})\pi x})}{\Lambda_{m-m_{1},j}\,\Lambda_{m-m_{1}-m_{2},j}}.
\end{equation}
\begin{equation}\label{R2}
R_{2}(m)=\sum_{m_{1},m_{2},m_{3}}\frac{q_{m_{1}}q_{m_{2}}q_{m_{3}}(q\,\Psi_{m,j},e^{i2(m-m_{1}-m_{2}-m_{3})\pi x})}{\Lambda_{m-m_{1},j}\,\Lambda_{m-m_{1}-m_{2},j}\,\Lambda_{m-m_{1}-m_{2}-m_{3},j}},
\end{equation}
\[
m_{i}\neq 0,\,\forall i;\quad \sum_{i=1}^{k}m_{i}\neq 0,2m,\qquad \forall k=1,2,3.
\]
Using (\ref{dist1}), (\ref{m3}) and the relation
\begin{equation}\label{main}
\sum_{m_{1}\neq 0,2m}\frac{1}{|m_{1}||2m-m_{1}|}=O\left(\frac{ln|m|}{m}\right)
\end{equation}
one can prove the estimates
\begin{equation}\label{m45}
R_{i}(m)=O\Big((\frac{ln| m|}{m})^{i+1}\Big),\;i=1,2.
\end{equation}
In the same way, by using the eigenfunction $e^{-i2m\pi
x}$ of
the operator $L$ for $q=0$, we can obtain
the relations
\begin{equation}\label{m413}
[\Lambda_{m,j}- a'(\lambda_{m,j})]v_{m,j}=[q_{-2m}+b'_{1}(\lambda_{m,j})]u_{m,j}+R_{1}'(m),
\end{equation}
\begin{equation}\label{m414}
[\Lambda_{m,j}- a'_{1}(\lambda_{m,j})-a'_{2}(\lambda_{m,j})]v_{m,j}=[q_{2m}+b'_{1}(\lambda_{m,j})+b'_{2}(\lambda_{m,j})]u_{m,j}+R'_{2}(m),
\end{equation}
where
\[
a'_{1}(\lambda_{m,j})=\sum_{m_{1}}\frac{q_{m_{1}}q_{-m_{1}}}{\Lambda_{m+m_{1},j}},\quad
a'_{2}(\lambda_{m,j})=\sum_{m_{1},m_{2}}\frac{q_{m_{1}}q_{m_{2}}q_{-m_{1}-m_{2}}}{\Lambda_{m+m_{1},j}\,\Lambda_{m+m_{1}+m_{2},j}},
\]
\begin{equation}\label{B1'}
b'_{1}(\lambda_{m,j})=\sum_{m_{1}}\frac{q_{m_{1}}q_{-2m-m_{1}}}{\Lambda_{m+m_{1},j}},\,
b'_{2}(\lambda_{m,j})=\sum_{m_{1},m_{2}}\frac{q_{m_{1}}q_{m_{2}}q_{-2m-m_{1}-m_{2}}}{\Lambda_{m+m_{1},j}\,\Lambda_{m+m_{1}+m_{2},j}},
\end{equation}
\[
  R'_{1}(m)=\sum_{m_{1},m_{2}}\frac{q_{m_{1}}q_{m_{2}}(q\,\Psi_{m,j},e^{i2(m+m_{1}+m_{2})\pi x})}{\Lambda_{m+m_{1},j}\,\Lambda_{m+m_{1}+m_{2},j}}.
\]
\begin{equation}\label{R2'}
R'_{2}(m)=\sum_{m_{1},m_{2},m_{3}}\frac{q_{m_{1}}q_{m_{2}}q_{m_{3}}(q\,\Psi_{m,j},e^{i2(m+m_{1}+m_{2}+m_{3})\pi x})}{\Lambda_{m+m_{1},j}\,\Lambda_{m+m_{1}+m_{2},j}\,\Lambda_{m+m_{1}+m_{2}+m_{3},j}},
\end{equation}
\[m_{i}\neq 0,\quad \sum_{i=1}^{k}m_{i}\neq 0,-2m,\quad \forall
k=1,2,3.\] Here the similar estimates as in (\ref{m45}) are valid
for $R'_{i}(m),$ $i=1,2.$

In addition, by using (\ref{dist1}), (\ref{m1}) and (\ref{m3}), we get
\[
\sum_{k\in  \mathcal{Z};\,k\neq \pm m}\Big|(\Psi_{m,j},e^{i2k\pi x})\Big|^{2}=O\left(\frac{1}{m^{2}}\right).
\]
Thus, we obtain that the normalized eigenfunctions  $\Psi_{m,j}(x)$ by the basis $\{e^{i2k\pi x}:k\in  \mathcal{Z}\}$ on $[0,1]$ has the following expansion
\begin{equation}\label{m7}
\Psi_{m,j}(x)=u_{m,j}\,e^{i2m\pi x}+v_{m,j}\,e^{-i2m\pi x}+h_{m}(x),
\end{equation}
where
\[
 (h_{m},e^{\mp i2m\pi x})=0,\quad \|h_{m}(x)\|=O(m^{-1}),
\]
\begin{equation}\label{m8}
  |u_{m,j}|^{2}+|v_{m,j}|^{2}=1+O\left(m^{-2}\right).
\end{equation}
Now, let us consider the following form of the Riemann-Lebesgue lemma. By this we set
\begin{equation}\label{m-8}
\rho(m)=:max\left\{\sup_{0\leq x\leq 1} \left| \int_{0}^{x}q(t)\,e^{-i2(2m)\pi t}dt\right|,\,\sup_{0\leq x\leq 1} \left| \int_{0}^{x}q(t)\,e^{i2(2m)\pi t}dt\right|\right\},
\end{equation}
and clearly $\rho(m)\rightarrow 0$ as $m\rightarrow \infty$. As the proof of lemma is similar to that of Lemma 6 in \cite{Harris:form}, %the proving lemma will be adequate to give in our notation
we pass to the proof.
\begin{lemma}\label{l1}
 If $q\in L^1[0,1]$ then $\displaystyle\int_{0}^{\,x}q(t)\,e^{i2m\pi t}dt\rightarrow 0$ as $|m|\rightarrow\infty$ uniformly in $x$.
\end{lemma}
\section{Main results}\label{results}
To prove the  main results of the paper we need the following lemmas.
\begin{lemma}\label{L1}
The eigenvalues $\lambda_{m,j}$ of the operator $L$ for $m\geq N$ and $j=1,2$,  satisfy
\begin{equation}\label{lem}
\lambda_{m,j}=(2m\pi)^{2}+O(\rho(m)),
\end{equation}
where $\rho(m)$ is defined in (\ref{m-8}).
\end{lemma}
\begin{proof} For the proof we have to estimate the terms of (\ref{m412}) and (\ref{m413}).
It is easily seen that
\begin{equation}\label{dif}
\sum_{m_{1}\neq 0,\pm2m}\left|\frac{1}{\Lambda_{m\mp m_{1},j}}-\frac{1}{\Lambda_{m\mp m_{1}}^{0}}\right|=O\left(\frac{\Lambda_{m,j}}{m^{2}}\right),
\end{equation}
where $\Lambda_{m\mp m_{1}}^{0}=(2m\pi)^{2}-(2(m\mp m_{1})\pi)^{2}$.
Thus, we get
\[
    a_{1}(\lambda_{m,j})=\frac{1}{4\pi^{2}}\sum_{m_{1}\neq 0,2m}\frac{q_{m_{1}}q_{-m_{1}}}{m_{1}(2m-m_{1})}+O\left(\frac{\Lambda_{m,j}}{m^{2}}\right).
\]
From the argument in Lemma 2(a) of \cite{veliev;arþiv} we deduce, with our notations,
\[
a_{1}(\lambda_{m,j})=\frac{1}{2\pi^{2}}\sum_{m_{1}> 0,m_{1}\neq2m}\frac{q_{m_{1}}q_{-m_{1}}}{(2m+m_{1})(2m-m_{1})}+O\left(\frac{\Lambda_{m,j}}{m^{2}}\right)
\]
\begin{equation}\label{d3}
=\int_{0}^{1}(G(x,m)-G_{0}(m))^{2}\,e^{i2(4m)\pi x}\,dx+O\left(\frac{\Lambda_{m,j}}{m^{2}}\right),
\end{equation}
where
\begin{equation}\label{d2}
G(x,m)=\int_{0}^{x}q(t)\,e^{-i2(2m)\pi t}dt-q_{2m}x,
\end{equation}
\begin{equation}\label{d4}
G_{m_{1}}(m)=:(G(x,m), e^{i2m_{1}\pi x})=\frac{q_{2m+m_{1}}}{i2\pi m_{1}}
\end{equation}
for $m_{1}\neq 0$ and
\[
G(x,m)-G_{0}(m)=\sum_{m_{1}\neq2m}\frac{q_{m_{1}}}{i2\pi(m_{1}-2m)}\,e^{i2(m_{1}-2m)\pi x}.
\]
Thus, from the equalities
\begin{equation}\label{gg}
     G(x,m)-G_{0}(m)=O(\rho(m)),\quad G(1,m)=G(0,m)=0
\end{equation}
(see (\ref{m-8}) and (\ref{d2})) and since $q\in L^1[0,a]$, integration by parts gives for the integral in (\ref{d3}) the estimate
\begin{equation}\label{s0}
     a_{1}(\lambda_{m,j})=O\left(\frac{\rho(m)}{m}\right)+O\left(\frac{\Lambda_{m,j}}{m^{2}}\right)
\end{equation}
for large $m$. % Here and in subsequent relations, we have written $\rho(m)$ for $\rho(2m+2)$.
 It is easily seen by substituting $m_{1}=-k$  into the relation for $a'_{1}(\lambda_{m,j})$ (see (\ref{m413})) that
\begin{equation}\label{a1}
     a_{1}(\lambda_{m,j})=a_{1}'(\lambda_{m,j}).
\end{equation}
In a similar way, by (\ref{dif}), etc., we get
\[b_{1}(\lambda_{m,j})=\frac{1}{4\pi^{2}}\sum_{m_{1}\neq 0,2m}\frac{q_{m_{1}}q_{2m-m_{1}}}{m_{1}(2m-m_{1})}+O\left(\frac{\Lambda_{m,j}}{m^{2}}\right)\]
\[
=-\int_{0}^{1}(Q(x)-Q_{0})^{2}\,e^{-i2(2m)\pi x}dx+O\left(\frac{\Lambda_{m,j}}{m^{2}}\right)
\]
\begin{equation}\label{I0}
\qquad\;\;=\frac{-1}{i2\pi(2m)}\int_{0}^{1}2(Q(x)-Q_{0})\,q(x)\,e^{-i2(2m)\pi x}dx+O\left(\frac{\Lambda_{m,j}}{m^{2}}\right),
\end{equation}
where  $Q(x)=\displaystyle\int_{0}^{x}q(t)\, dt,\quad Q_{m_{1}}=:(Q(x),e^{i2m_{1}\pi x})=\frac{q_{m_{1}}}{i2\pi m_{1}} \quad \textrm{if}\,\, m_{1}\neq 0$,
\begin{equation}\label{Q0}
Q(x)-Q_{0}=\sum_{m_{1}\neq 0}Q_{m_{1}}\,e^{i2m_{1}\pi x}.
\end{equation}
Thus, by using $Q(1)=q_{0}=0$  and (\ref{m-8}), integration by parts again gives for the integral in (\ref{I0}) the following estimate
\begin{equation}\label{b1}
    b_{1}(\lambda_{m,j})=O\left(\frac{\rho(m)}{m}\right)+O\left(\frac{\Lambda_{m,j}}{m^{2}}\right).
\end{equation}
Similarly
\begin{equation}\label{b1'}
    b'_{1}(\lambda_{m,j})=O\left(\frac{\rho(m)}{m}\right)+O\left(\frac{\Lambda_{m,j}}{m^{2}}\right).
\end{equation}
To estimate $R_{1}(m)=o\left(\rho(m)\right)$ (see (\ref{R})), let us show that
\begin{equation}\label{rhoineq}
    \rho(m)>C\,m^{-1}
\end{equation}
for $m\geq N$ and some $C>0$. Since $q(x)\neq 0$ is summable function on $[0,1]$, there exists $x\in[0,1]$ such that
\begin{equation}\label{qint}
    \int_{0}^{x}q(t)\,dt\neq 0
\end{equation}
and the integral (\ref{qint}) is bounded for all $x\in [0,1].$
Hence, multiplying  the integrand of (\ref{qint}) by  $e^{-i2(2m)\pi x}e^{i2(2m)\pi x}$, and then using integration by parts, we get
\[
    \sup_{0\leq x\leq 1}\left|\int_{0}^{x}q(t)\,dt\right|\leq C\left(\rho(m)+m\rho(m)\right)\leq Cm\rho(m)
\]
which implies (\ref{rhoineq}).

Thus by (\ref{dist1}), (\ref{m3}) and relation (\ref{main}),
we deduce that
\[
|R_{1}(m)|\leq C\frac{(ln|m|)^{2}}{m^{2}}=o\left(\rho(m)\right).
\]
Also $R'_{1}(m)=o\left(\rho(m)\right).$

From the relation (\ref{m8}), for large $m$, it follows that either $|u_{m,j}|>1/2$ or $|v_{m,j}|>1/2$. We first consider the case when $|u_{m,j}|>1/2$. Hence, by using (\ref{m412}), (\ref{s0}) and (\ref{b1}) with $R_{1}(m)=o\left(\rho(m)\right)$ we obtain
\[
\Lambda_{m,j}(1+O(m^{-2}))=q_{2m}\frac{v_{m,j}}{u_{m,j}}+o\left(\rho(m)\right).
\]
This with the definition (\ref{m-8}) gives $\Lambda_{m,j}=O\left(\rho(m)\right)$. Similarly, for the other case $|v_{m,j}|>1/2$, by using (\ref{m413}), (\ref{s0}), (\ref{b1'}) and $R_{1}'(m)=o\left(\rho(m)\right)$, we get (\ref{lem}). The lemma is proved.
\qed
\end{proof}
\begin{lemma}\label{L2}
For all large $m$, we have the following estimates (see, respectively, (\ref{B1}), (\ref{B1'}) and (\ref{R2}), (\ref{R2'}))
\begin{equation}\label{Rlmma}
 b_{2}(\lambda_{m,j}),\, b'_{2}(\lambda_{m,j})=O\left(\rho(m)m^{-2}\right),\quad R_{2}(m),\,R'_{2}(m)=O\left(\rho(m)m^{-1}\right).
\end{equation}
\end{lemma}
\begin{proof}
Let us estimate the sum $R_{2}(m)$. By using the estimate (\ref{m45}) and the inequality (\ref{rhoineq}) for large $m$, we deduce that
\[
|R_{2}(m)|\leq C\frac{(ln|m|)^{3}}{m^{3}}=O\left(\rho(m)m^{-1}\right).
\]
In the same way $R'_{2}(m)=O\left(\rho(m)m^{-1}\right)$.

Arguing as in \cite{Veliev:Shkalikov} (see the
proof of Lemma 6), let us now estimate the sum $b_{2}(\lambda_{m,j})$. Taking into account (\ref{dif}) and Lemma \ref{L1}, we have
\begin{equation}\label{b2}
    b_{2}(\lambda_{m,j})=\frac{1}{(2\pi)^{4}}I(m)+O\left(\frac{\rho(m)}{m^{3}}\right),
\end{equation}
where
\[
I(m)=\sum_{m_{1},m_{2}}\frac{q_{m_{1}}q_{m_{2}}q_{2m-m_{1}-m_{2}}}{m_{1}(2m-m_{1})(m_{1}+m_{2})(2m-m_{1}-m_{2})}.
\]
By using the identity
\[
  \frac{1}{k(2m-k)}=\frac{1}{2m}\left(\frac{1}{k}+\frac{1}{2m-k}\right),
\]
and the substitutions $k_{1}=m_{1}$, $k_{2}=2m-m_{1}-m_{2}$ in the formula $I(m)$, we obtain $I(m)$ with the indices $m_{1}, m_{2}$ in the following form
\begin{equation}\label{equal2}
   I(m)=\frac{1}{(2m)^{2}}(I_{1}+2I_{2}+I_{3}),
\end{equation}
where\[
I_{1}=\sum_{m_{1},m_{2}}\frac{q_{m_{1}}q_{m_{2}}q_{2m-m_{1}-m_{2}}}{m_{1}m_{2}},\quad I_{2}=\sum_{m_{1},m_{2}}\frac{q_{m_{1}}q_{m_{2}}q_{2m-m_{1}-m_{2}}}{m_{2}(2m-m_{1})},
\]
\[
I_{3}=\sum_{m_{1},m_{2}}\frac{q_{m_{1}}q_{m_{2}}q_{2m-m_{1}-m_{2}}}{(2m-m_{1})(2m-m_{2})}.
\]
From (\ref{d4})-(\ref{gg}), (\ref{Q0}), $2I_{2}(m)=I_{1}(m)$ and using integration by parts only in $I_{1}$, we obtain the following estimates
\begin{equation}\label{I123}
\left.
\begin{array}{ll}
\displaystyle I_{1}=-4\pi^{2}\int_{0}^{1}(Q(x)-Q_{0})^{2}\,q(x)\,e^{-i2(2m)\pi x}dx=O\left(\rho(m)\right), &  \\\\
\displaystyle  I_{3}=-4\pi^{2}\int_{0}^{1}(G(x,m)-G_{0}(m))^{2}\,q(x)\,e^{i2(2m)\pi x}dx=O\left(\rho(m)\right). &
\end{array}
    \right\}
\end{equation}
Then, in view of (\ref{equal2}) and (\ref{I123}), $I(m)=O\left(\rho(m)m^{-2}\right).$ This with the equality (\ref{b2})
implies that $ b_{2}(\lambda_{m,j})=O\left(\rho(m)m^{-2}\right).$ In the same way $ b'_{2}(\lambda_{m,j})$ satisfies the same estimate. The lemma is proved.
\qed
\end{proof}
Thus by using Lemma \ref{L1}-\ref{L2}, Theorem \ref{velievthm} and an argument similar to that of Theorem 2 in \cite{Veliev:Shkalikov} under the conditions (\ref{maincon}),  let us prove the following main result.
\begin{flushleft}
    \textbf{Proof of Theorem \ref{mainthm}}
\end{flushleft}
In view of Lemma \ref{L1}, substituting the values of
\[ b_{1}(\lambda_{m,j}),\,  b'_{1}(\lambda_{m,j})=O\left(\rho(m)m^{-1}\right),\]
\[ b_{2}(\lambda_{m,j}),\,  b'_{2}(\lambda_{m,j})R_{2}(m),\, R'_{2}(m)=O\left(\rho(m)m^{-2}\right)\]
given by (\ref{b1}), (\ref{b1'}), (\ref{Rlmma})  in the relations (\ref{m4123}) and (\ref{m414}), we get the following reversion of the relations
\begin{equation}\label{son}
\left[\Lambda_{m,j}-a_{1}(\lambda_{m,j})-a_{2}(\lambda_{m,j})\right]u_{m,j}=\left[q_{2m}+O\left(\rho(m)m^{-1}\right)\right]v_{m,j}+O(\rho(m)m^{-2}),
\end{equation}
\begin{equation}\label{son'}
\left[\Lambda_{m,j}-a'_{1}(\lambda_{m,j})-a'_{2}(\lambda_{m,j})\right]v_{m,j}=\left[q_{-2m}+O\left(\rho(m)m^{-1}\right)\right]u_{m,j}+O(\rho(m)m^{-2})
\end{equation}
for $j=1,2$.

It is easily seen again by substituting $m_{1}+m_{2}=-k_{1}$, $m_{2}=k_{2}$ in the sum $a_{2}'(\lambda_{m,j})$ (see (\ref{m413})) and using (\ref{a1})  that $a_{i}(\lambda_{m,j})=a'_{i}(\lambda_{m,j})$ for $i=1,2$. Hence, multiplying (\ref{son}) by $v_{m,j}$ and (\ref{son'})
by $u_{m,j}$ and subtracting we obtain the following equality
\begin{equation}\label{uvfark}
q_{2m}\,v_{m,j}^{2}-q_{-2m}\,u_{m,j}^{2}=O(\rho(m)m^{-1}).
\end{equation}

Suppose, for example, that $q_{2m}$ satisfies the condition in (\ref{maincon}). Then using this equality we get
\begin{equation}\label{uv2fark}
v_{m,j}^{2}-\kappa_{m}u_{m,j}^{2}=o(1),\qquad\kappa_{m}=:\frac{q_{-2m}}{q_{2m}}.
\end{equation}
for $j=1,2$.  In addition, for large $m$, the condition (\ref{maincon}) for $q_{2m}$ implies that the geometric multiplicity of the eigenvalue $\lambda_{m,j}$ is 1. Arguing as in Lemma 4 of \cite{Veliev:Shkalikov}, if there exist mutually orthogonal two eigenfunctions $\Psi_{m,j}(x)$ corresponding to $\lambda_{m,1}=\lambda_{m,2}$, then one can choose an eigenfunction $\Psi_{m,j}(x)$ such that $u_{m,j}=0$. Thus combining this with (\ref{m8}) and (\ref{uvfark}), we get $q_{2m}=O(\rho(m)m^{-1})$
which contradicts (\ref{maincon}).

Let the normal system of root functions form a Riesz basis. To prove $\kappa_{m}\sim 1$, from (\ref{uv2fark}) it is enough to show that all the large periodic eigenvalues $\lambda_{m,j}$ are simple,
since in this case we have, by Theorem \ref{velievthm},
\begin{equation}\label{vsimilar}
u_{m,j}\sim v_{m,j}\sim 1
\end{equation}
for $j=1,2$. For large $m$, again by Theorem \ref{velievthm} and the condition (\ref{maincon}) for $q_{2m}$, respectively, the number of Jordan chains and the eigenvalues of geometric multiplicity $2$ is finite, that is,
all large eigenvalues are simple.

Now let $q_{2m}\sim q_{-2m}$. From the second formula of (\ref{uv2fark}), we obtain that $\kappa_{m}\sim 1$ and then, from the first, that (\ref{vsimilar}) for the eigenfunction $\Psi_{m,1}(x)$, that is, $j=1$ which implies that the number of Jordan chains is finite. In fact, if there exists a Jordan chain consists of an eigenfunction $\Psi_{m,1}(x)$ and an associated function $\Psi_{m,2}(x)$ corresponding to the eigenvalue $\lambda_{m,1}=\lambda_{m,2}$, then, for example for $\lambda_{m,1}$, using the eigenfunction $\overline{\Psi_{m,1}(x)}$ of the adjoint operator $L^{*}$ and the relation
\[
(L-\lambda_{m,1})\Psi_{m,2}(x)=\Psi_{m,1}(x),
\]
we obtain that $(\Psi_{m,1},\overline{\Psi_{m,1}})=0$. Thus, from the expansion (\ref{m7}) for $j=1$, we get $u_{m,1}v_{m,1}=O(m^{-2})$ which contradicts (\ref{vsimilar}) for $j=1$. Thus, using Theorem \ref{velievthm},
we prove that a normal system of root functions of the operator $L$ forms a Riesz basis.
\qed

Arguing as in the proof of Theorem \ref{mainthm}, we obtain a similar result established below for the anti-periodic problems.
\begin{theorem}\label{mainthmanti}
Let $q\in L_{1}[0,1]$ be arbitrary complex-valued function and suppose that at least one of the conditions
\[
\lim_{m\rightarrow\infty}\frac{\rho(m)}{m\,q_{2m+1}}=0,\qquad\lim_{m\rightarrow\infty}\frac{\rho(m)}{m\,q_{-2m-1}}=0
\]
is satisfied, where $\rho(m)$ is obtained from (\ref{m-8}) by replacing $2m$ with $2m+1$ and a common order of both Fourier coefficients $q_{2m+1}$ and $q_{-2m-1}$ of $q$.

Then a normal system of root functions of the operator $L$ with anti-periodic boundary conditions forms a Riesz basis if and only if    $ q_{2m+1}\sim q_{-2m-1}.$
\end{theorem}
\begin{remark}
Clearly if instead of (\ref{maincon}) we assume that at least one of the conditions
\[
\rho(m)\sim q_{2m},\qquad \rho(m)\sim q_{-2m}
\]
holds, then the assertion of Theorem \ref{mainthm} is satisfied. In this way one can easily write a similar result for the anti-periodic problem.

In addition to all the above results, we note that if either the first condition of (\ref{maincon}) and (\ref{mainq2mcond}), or the second condition of (\ref{maincon}) and (\ref{mainq2mcond}) hold then all the periodic eigenvalues are asymptotically simple. We can write a similar result for the anti-periodic problem.
\end{remark}

\bibliographystyle{spmpsci}      % mathematics and physical sciences
%\bibliography{article}   % name your BibTeX data base


\end{document}